\documentclass[12pt,reqno]{amsart}
\usepackage{soul}
\usepackage{multicol,amsmath,graphics, mathtools, cancel,soul,color,bbm,amsfonts,dsfont,tikz,mathrsfs,amssymb,bm,bbold}
\usepackage{verbatim}
\usepackage[left=1in, right=1in, top=1.1in,bottom=1.1in]{geometry}
\usetikzlibrary{matrix,patterns,positioning}
\numberwithin{equation}{section}
 
\newcommand{\E}{\mathbb{E}}

\newcommand{\N}{\mathbb{N}}
\newcommand{\Pb}{\mathbb{P}}

\newcommand{\R}{\mathbb{R}}

\newcommand{\vertiii}[1]{{\left\vert\kern-0.25ex\left\vert\kern-0.25ex\left\vert #1
    \right\vert\kern-0.25ex\right\vert\kern-0.25ex\right\vert}}

\def\R{\mathbb{R}}

\def\dh2l{\mathbf{d}_{\mathbb{H}_{2\ell}}}
\def\d2{\mathbf{d}_2}

\newtheorem{thm}{Theorem}[section]
\newtheorem{lemma}[thm]{Lemma}

\newtheorem{definition}[thm]{Definition}

\newtheorem{theorem}{Theorem}[section]
\newtheorem{corollary}[theorem]{Corollary}

\theoremstyle{remark}
\newtheorem{remark}[theorem]{Remark}
\theoremstyle{definition}

\newcounter{dummy} \numberwithin{dummy}{section}

\newtheorem{Proposition}[dummy]{Proposition}

\newtheorem{Lemma}[dummy]{Lemma}

\def\1{\mathbbm{1}}

\begin{document}
\title{Non-commutative law of rare events}

\author{Marco Tulio Gaxiola}
\address{Marco Tulio Gaxiola: Ometra, Mexico City.}
\email{tgaxiola@ometra.mx}

\author{Arturo Jaramillo}
\address{Arturo Jaramillo: Department of Probability and Statistics,
Centro de Investigaci\'on en Matem\'aticas (CIMAT)}
\email{jagil@cimat.mx}

\date{\today}
\begin{abstract}
We establish quantitative versions of the law of rare events and binomial approximations in non-commutative probability settings, including the free, Boolean, and monotone convolution frameworks. Our main results provide explicit error bounds in the non-commutative Wasserstein distance for approximations of convolutions of rare countings by non-commutative Poisson and binomial distributions. These bounds extend classical results from the tensor setting to the non-commutative regime. Our approach relies on a discrete Lindeberg-type interpolation scheme combined with algebraic properties of cumulants adapted to each independence notion. The results presented here fill a gap in the literature concerning explicit rates of convergence in non-commutative limit theorems.
\end{abstract}

\maketitle

\section{Introduction}
Among the most fundamental results in the study of distributional limit theorems is the law of rare events, which establishes conditions ensuring asymptotic Poissonian behavior for sums of independent, not necessarily identically distributed, Bernoulli random variables. This problem, together with quantitative estimates of the associated approximation error, has been central to a substantial body of research on limit theorems in discrete settings.\\

\noindent This work extends the classical law of rare events into the realm of non-commutative probability. Our main contribution is to establish quantitative bounds for the approximation error in Poisson-type limit theorems under free and Boolean convolutions, with the error measured in terms of the non-commutative Wasserstein distance. These results complement the well-developed theory available in the classical (tensor) framework. In addition, the convolution stability argument yields analogous binomial approximation bounds in the tensor and free settings.\\

\noindent The free and Boolean cases require different arguments. The free approximation is based on a discrete Lindeberg-type interpolation, in which the Bernoulli distributions are successively replaced by free Poisson distributions of the corresponding intensities. The Boolean case is treated directly by combining an explicit support estimate with a classical coupling adapted to the two-point structure of the Boolean Poisson distribution. In both cases, elementary cumulant identities provide the moment information needed to obtain the final estimates.\\

\noindent \textit{Historical Overview}\\
The foundations of Poisson approximation theory trace back to the seminal work of Le Cam~\cite{LeCam1960}, who first established a rate of convergence, in total variation distance, for the approximation of the distribution of a sum of \( n \) independent Bernoulli random variables with small success probabilities by a suitably chosen Poisson distribution. The field underwent a significant methodological breakthrough with the introduction of Stein's method~\cite{MR402873}, which reinterprets probability distances as expectations of difference operators applied to test functions and integrated with respect to the target distribution. This perspective opened the door to powerful new tools for quantitative approximation. Subsequent developments further contributed to the theory, including the introduction of models with local dependencies via dependency graphs, as in the work of Arratia, Goldstein, and Gordon~\cite{Arratia1989}. A pedagogical and comprehensive introduction to Stein's method for Poisson approximation is provided in~\cite{ChenGoldsteinShao2011}. The general type of results found in the aforementioned works takes the form:
\begin{align}\label{eq:Lawofrareevents1}
d_{TV} ( \left(\delta_0 (1 - p_{1,n}) + \delta_1 p_{1,n} \right) * \cdots * \left(\delta_0 (1 - p_{n,n}) + \delta_1 p_{n,n} ),\, \nu_n \right)
\leq C \sum_{k=1}^n p_{k,n}^2,
\end{align}
where \( \nu_n \) is the Poisson distribution with intensity \(  p_{1,n}+\cdots+p_{n,n} \), and \( C \) is a universal constant.\\

\noindent Analogues of the classical law of rare events are known in the free and Boolean convolution frameworks. Convergence of rare Bernoulli trials to the corresponding non-commutative Poisson distributions was established in~\cite{MR839105} for the free case and in~\cite{SpeicherWoroudi1997} for the Boolean case. To the best of our knowledge, explicit finite-sample bounds in the Wasserstein distances considered here for the approximation of Bernoulli convolutions by the corresponding free and Boolean Poisson distributions have not previously been established. A more detailed discussion of non-commutative Poisson distributions and their relevant properties is provided in Section~\ref{sec:ncpoissonprelim}. In the Boolean case, we prove a bound of the same order as~\eqref{eq:Lawofrareevents1}, while in the free setting we obtain an estimate in which \(p_{k,n}^2\) is replaced by \(p_{k,n}^{3/2}\). It remains open whether the exponent \(3/2\) in the free case can be replaced by \(2\).\\

\noindent The precise formulation of our main results involves a substantial amount of notation and is therefore deferred to Section~\ref{Seq:noncommutativelawofrareevents}. Theorem~\ref{eq:lawofrareevents} provides quantitative Poisson approximation bounds for tensor, free, and Boolean convolutions of Bernoulli distributions. Corollary \ref{eq:lawofrareeventsgeneral} extends the tensor and free estimates to compactly supported probability measures that are close to Bernoulli distributions. Finally, Proposition~\ref{eq:binomial} provides analogous binomial approximation bounds in the tensor and free settings.\\

\noindent The remainder of the paper is organized as follows. Section~\ref{sec:prelim} introduces the necessary preliminaries and the main tools used throughout. Section~\ref{Seq:noncommutativelawofrareevents} presents the formal statements of our main results. Section~\ref{Proofs} is devoted to the development of the corresponding proofs. Finally, in the appendix we present a technical lemma that is used throughout the proof of the main result.

\section{Preliminaries}\label{sec:prelim}
\subsection{Elements of non-commutative probability}  
In this section, we provide a brief overview of the fundamental notions of independence in non-commutative probability. Intuitively, these concepts involve setting up a suitable structure in the joint law of non-commutative random variables \( X_1, \dots, X_r \), allowing for the computation of expectations of the form  
\[
\E[f_1(X_1)\cdots f_r(X_r)],
\]  
where \( f_1, \dots, f_r \) are polynomials in \( \mathbb{C}[X] \). The key idea is that this expectation can be determined purely from the marginal distributions of \( X_1, \dots, X_r \), without needing to consider their full joint distribution explicitly.\\  

\noindent A major breakthrough in this area was the discovery of free independence, which emerged from the study of quantum probability. This concept was first observed in the works of Hudson and Parthasarathy \cite{MR0745686} and later formalized by Voiculescu \cite{MR0799593}. Over time, efforts to establish a more general framework for independence led to the formulation of another natural notion of independence, Boolean independence \cite{MR1426845}. These notions of independence are formulated within the framework of non-commutative probability spaces, which we define next. Rather than treating random variables as functions on a measure space, we consider them as elements of a unital \( C^{*} \)-algebra, which we denote by \( \mathcal{A} \). This space is assumed to be equipped with a positive unital linear functional \( \tau: \mathcal{A} \to \mathbb{C} \). Any element \( a \in \mathcal{A} \) will be referred to as a non-commutative random variable. We say that \( a \) is self-adjoint if it satisfies \( a = a^* \).\\  

\noindent To describe the joint law of a collection of $d$ non-commutative random variables, we consider the space \( \mathfrak{D}_d \), consisting of linear functionals acting on the set of polynomials \( \mathbb{C}\langle X_1, \dots, X_d\rangle \) in \( d \) non-commutative variables and taking values in \( \mathbb{C} \). Given elements \( a_1, \dots, a_d \in \mathcal{A} \), the corresponding algebraic distribution is defined as the unique element \( \mathfrak{f}_{a_1, \dots, a_d} \in \mathfrak{D}_d \) that satisfies  
\begin{align*}
\mathfrak{f}_{a_1,\dots, a_d}[Q]
  &:=\tau[Q(a_1,\dots, a_d)],
\end{align*}
for every polynomial \( Q \in \mathbb{C}\langle X_1, \dots, X_d\rangle \).  This notion of algebraic distribution for a non-commutative random variable aligns with the classical definition in the case where \( a \) is self-adjoint. Specifically, there exists a unique compactly supported probability measure \( \mu_a \) (referred to as its analytic distribution) that has the same moments as \( a \), meaning that  
\begin{align*}
\int_\R Q(x)\mu_a(dx)
  &:=\mathfrak{f}_a[Q].
\end{align*}

\noindent Just as in classical probability, knowing the individual distributions of two self-adjoint elements \(a,b\in\mathcal{A}\) is not sufficient to determine their joint law. The missing information can be specified through an appropriate notion of independence. In this paper, we consider tensor, Boolean, and free independence, which we recall next.

\begin{definition}[Tensor independence]
A collection of unital subalgebras
\[
\{\mathcal{A}_k\ ;\ 1\leq k\leq r\}
\]
is said to be tensor independent if the algebras commute pairwise and
\[
\tau[a_1\cdots a_r]
=
\prod_{k=1}^r\tau[a_k]
\]
for every choice of elements \(a_k\in\mathcal{A}_k\).
\end{definition}

\begin{definition}[Boolean independence]
A collection of not necessarily unital subalgebras
\[
\{\mathcal{A}_k\ ;\ 1\leq k\leq r\}
\]
is said to be Boolean independent if
\[
\tau[a_1\cdots a_m]
=
\prod_{j=1}^m\tau[a_j]
\]
whenever \(a_j\in\mathcal{A}_{i(j)}\) and
\[
i(j)\neq i(j+1),
\qquad
1\leq j\leq m-1.
\]
\end{definition}

\begin{definition}[Free independence]
A collection of unital subalgebras
\[
\{\mathcal{A}_k\ ;\ 1\leq k\leq r\}
\]
is said to be freely independent if
\[
\tau[a_1\cdots a_m]=0
\]
whenever \(a_j\in\mathcal{A}_{i(j)}\),
\[
\tau[a_j]=0,
\qquad
1\leq j\leq m,
\]
and
\[
i(j)\neq i(j+1),
\qquad
1\leq j\leq m-1.
\]
\end{definition}

\subsection{Non-commutative Kantorovich-Rubenstein distance}\label{Kantorovichdubenstein}
Let $\mathcal{P}_c(\R)$ denote the set of compactly supported probability measures on $\R$. Let $\mathfrak{L}_2$ denote the set of linear functionals
\[
\rho:\mathbb{C}\langle X,Y\rangle\longrightarrow\mathbb{C}
\]
of the form $\rho=\mathfrak{f}_{a,b}$, where $a$ and $b$ are self-adjoint random variables in a tracial \(C^*\)-probability space. For a given element \( f \in \mathbb{C}[X] \), we define the polynomials \( f_X, f_Y, f_{Y-X} \in \mathbb{C}\langle X,Y\rangle \) as follows:  
\[
f_X(X,Y) := f(X), \quad f_Y(X,Y) := f(Y), \quad \text{and} \quad f_{Y-X}(X,Y) := f(Y-X).
\]
For any functional \( \rho \in \mathfrak{L}_2 \), we introduce the functionals \( \rho_X, \rho_Y, \rho_{Y-X} \) mapping \( \mathbb{C}[X] \) to \( \mathbb{C} \) and defined by  
\begin{align*}
\rho_X[f]  :=\rho[f_X],
\ \ \ \ \ \  
\rho_Y[f]  :=\rho[f_Y], 
\ \ \ \ \ \
\rho_{Y-X}[f]  :=\rho[f_{Y-X}].
\end{align*}
Since \(a\), \(b\), and \(b-a\) are bounded and self-adjoint, the functionals \(\rho_X\), \(\rho_Y\), and \(\rho_{Y-X}\) extend by continuous functional calculus to continuous functions on compact intervals containing their respective spectra. Given two probability measures \( \gamma_1, \gamma_2 \in \mathcal{P}_c(\mathbb{R}) \), we define the set of transport plans between them as  
\begin{align*}
\Pi[\gamma_1,\gamma_2]
  &:=\{\rho\in\mathfrak{L}_2\ ;\ \rho_X=\gamma_1\ \ \text{ and }\ \rho_Y=\gamma_2\}.	
\end{align*}

\begin{definition}
For \( p \in [1,\infty) \), the \( p \)-Kantorovich-Rubenstein distance is given by  
\begin{align*}
d_{W^{p}}(\gamma_1,\gamma_2)
  &:=\inf_{\rho\in\Pi[\gamma_1,\gamma_2]}\rho_{Y-X}[r_{p}]^{1/p},	
\end{align*}
where \( r_p(z) := |z|^{p} \).
\end{definition}
This definition is motivated by its classical counterpart in optimal transport, which we will refer to as the tensor Wasserstein distance. For probability measures \(\gamma_1\) and \(\gamma_2\) with finite moments of order \(p\), it is defined by
\begin{align}
\mathbbm{d}_{W^{p}}(\gamma_1,\gamma_2)
  &=\left(\inf_{\pi\in \mathbb{\Pi}[\gamma_1,\gamma_2]}
  \int_{\mathbb{R}^{2}} |x-y|^{p} \, \pi(dx,dy)\right)^{1/p},
\end{align}
where \( \mathbb{\Pi}[\gamma_1,\gamma_2] \) denotes the set of tensor transport plans in \( \mathbb{R}^2 \), meaning the collection of probability measures on \( \mathbb{R}^2 \) whose first and second marginals are \( \gamma_1 \) and \( \gamma_2 \), respectively. Since classical probability spaces can be regarded as a special case of non-commutative ones, this naturally leads to the inequality  
\begin{align}
d_{W^{p}}(\gamma_1,\gamma_2)
  &\leq \mathbbm{d}_{W^{p}}(\gamma_1,\gamma_2).	
\end{align}

For $\circledast\in\{*,\boxplus\}$, the following inequality will play a crucial role in our computations.

\begin{Lemma}\label{Lemma:convolutioninequality}
Let $p\in[1,\infty)$ and
$\circledast\in\{*,\boxplus\}$. Let
$\gamma_1,\rho_1,\dots,\gamma_n,\rho_n
\in\mathcal{P}_c(\mathbb{R})$ be probability measures with
finite moments of order $p$. Then
\begin{align}\label{eq:keyineq1}
d_{W^{p}}
\left(
\gamma_1\circledast\cdots\circledast\gamma_n,
\rho_1\circledast\cdots\circledast\rho_n
\right)
&\leq
\sum_{k=1}^{n}
d_{W^{p}}(\gamma_k,\rho_k).
\end{align}
\end{Lemma}

\subsection{Convolutions and cumulants}\label{sec:cumulants}

In this section, we consider the convolutions arising from tensor, free, and Boolean independence. Given two probability measures \( \mu \) and \( \nu \), we can construct a non-commutative probability space \( (\mathcal{A},\tau) \) along with self-adjoint elements \( a \) and \( b \) in \( \mathcal{A} \), whose generated algebras satisfy the corresponding independence condition. Since the sum \( a+b \) remains self-adjoint, it has a well-defined analytic distribution, which we denote by \( \mu_{a+b} \). For $\circledast\in\{*,\boxplus,\uplus\}$, corresponding respectively to tensor, free, and Boolean convolution, we use the notation \( \mu\circledast\nu \) for the distribution \( \mu_{a+b} \). The admissible choices of \(\circledast\) will be specified in each statement.\\

\noindent\textit{Cumulants}\\
For a probability measure $\gamma\in\mathcal{P}(\R)$ with
finite moments, we use the notation
\[
m_r[\gamma]:=\int_{\R}x^r\gamma(dx).
\]
The cumulants associated with the different notions of independence are defined through the corresponding moment-cumulant formulas.\\

\noindent Throughout this section, we write $[r]:=\{1,\dots,r\}$. Let $\mathcal{P}(r)$ denote the set of all partitions of $[r]$, let $NC(r)$ denote the set of non-crossing partitions of $[r]$, and let $I(r)$ denote the set of interval partitions of $[r]$.

\begin{definition}
Let $\gamma\in\mathcal{P}(\R)$ have moments of all orders.
For $\circledast\in\{*,\boxplus,\uplus\}$, the
$\circledast$-cumulants of $\gamma$ are the unique sequence
$\{\kappa_r^{\circledast}[\gamma]\ ;\ r\geq1\}$ satisfying
\[
m_r[\gamma]
=
\sum_{\mathfrak{p}\in\mathcal{P}_{\circledast}(r)}
\prod_{V\in\mathfrak{p}}
\kappa_{|V|}^{\circledast}[\gamma],
\]
where
\[
\mathcal{P}_{*}(r)=\mathcal{P}(r),\qquad
\mathcal{P}_{\boxplus}(r)=NC(r),\qquad
\mathcal{P}_{\uplus}(r)=I(r).
\]
\end{definition}

\noindent These formulas define the cumulants recursively in terms of the moments. The classical formula is the usual moment-cumulant formula, while the free and Boolean formulas may be found in \cite[Section 2.2]{ArizmendiCelestino2022}.\\

\noindent For $y\in\R$, let $\mathfrak{D}_y:\mathcal{P}(\R)\rightarrow\mathcal{P}(\R)$ be the dilation operator defined by
\[
\int_{\R}f(x)\mathfrak{D}_y[\gamma](dx)
=
\int_{\R}f(yx)\gamma(dx)
\]
for every bounded continuous function $f$.\\

\noindent The properties of cumulants that will be used in this paper
are collected in the following lemma.

\begin{lemma}\label{lem:cumulantproperties}
Let $\gamma,\rho\in\mathcal{P}(\R)$ have moments of all
orders.

\begin{enumerate}
\item For every $\circledast\in\{*,\boxplus,\uplus\}$ and $r\geq1,$
\[
\kappa_r^{\circledast}[\gamma\circledast\rho]
=
\kappa_r^{\circledast}[\gamma]
+
\kappa_r^{\circledast}[\rho],
\]

\item For every $\circledast\in\{*,\boxplus,\uplus\}$ and
$y\in\R$,
\[
\kappa_r^{\circledast}[\mathfrak{D}_y[\gamma]]
=
y^r\kappa_r^{\circledast}[\gamma].
\]
\end{enumerate}
\end{lemma}

\begin{proof}
The classical statements are standard. The free and Boolean
statements follow directly from the corresponding
moment-cumulant formulas; see \cite[Section 2.2]{ArizmendiCelestino2022}.
\end{proof}
The first two moments can be written in terms of cumulants in the following form, regardless of the choice of $\circledast$:
\[
m_1=\kappa_1^{\circledast},
\qquad
m_2
=
\left(\kappa_1^{\circledast}\right)^2
+
\kappa_2^{\circledast}.
\]
For the third moment in the Boolean and free cases, we have
\begin{align}\label{eq:momentoforderthreetocumulants}
m_3
&=
\left(\kappa_1^{\uplus}\right)^3
+
2\kappa_2^{\uplus}\kappa_1^{\uplus}
+
\kappa_3^{\uplus}
\\
&=
\left(\kappa_1^{\boxplus}\right)^3
+
3\kappa_2^{\boxplus}\kappa_1^{\boxplus}
+
\kappa_3^{\boxplus}.
\nonumber
\end{align}
These identities follow directly from general moment-cumulant formulas; see \cite[Section 2.2]{ArizmendiCelestino2022}.

\subsection{Non-commutative Poisson distribution}\label{sec:ncpoissonprelim}
In this section we introduce the notion of non-commutative Poisson distribution. 

\begin{definition}
We say that a probability measure
$\mu\in\mathcal{P}(\R)$ with moments of all orders has a $\circledast$-Poisson distribution with intensity $\lambda\geq0$ if
\begin{align*}
\kappa_r^{\circledast}[\mu]
  &=\lambda
\end{align*}
for every $r\in\N$. We denote this probability measure by
$\Pi_{\lambda}^{\circledast}$. In particular, we set
\[
\Pi_0^{\circledast}:=\delta_0.
\]
\end{definition}

For every $\circledast\in\{*,\boxplus,\uplus\}$, the additivity of cumulants implies that
\[
\Pi_{\lambda_1}^{\circledast}
\circledast
\Pi_{\lambda_2}^{\circledast}
=
\Pi_{\lambda_1+\lambda_2}^{\circledast}.
\]
The $\circledast$-Poisson distribution is crucial in the study of non-commutative convolutions, due to the law of rare events, presented in the next theorem  
\begin{theorem}
Let $\circledast\in\{*,\boxplus,\uplus\}$ and let $\lambda>0$. For every integer $n>\lambda$, define the Bernoulli probability measure
\[
\nu_n
=
\left(1-\frac{\lambda}{n}\right)\delta_0
+
\frac{\lambda}{n}\delta_1.
\]
Then $\nu_n^{\circledast n}$  converges weakly to the $\circledast$-Poisson distribution
$\Pi_{\lambda}^{\circledast}$.
\end{theorem}
In the tensor case, a quantitative version of the law of
rare events was established in \cite{MR428387}. The
classical Poisson distribution is supported on
$\N_0:=\N\cup\{0\}$ and satisfies
\[
\Pi_\lambda^*[\{k\}]
=
e^{-\lambda}\frac{\lambda^k}{k!},
\qquad
k\in\N_0.
\]
In the free case, the problem was addressed in
\cite{MR839105}. The distribution
$\Pi_\lambda^{\boxplus}$ has an atom at zero of size
$(1-\lambda)_+$ and an absolutely continuous part supported
on
\[
[(1-\sqrt{\lambda})^2,(1+\sqrt{\lambda})^2],
\]
with density
\[
\frac{1}{2\pi x}
\sqrt{
\big((1+\sqrt{\lambda})^2-x\big)
\big(x-(1-\sqrt{\lambda})^2\big)
}.
\]
The Boolean case was addressed in
\cite{MR1426845}. The Boolean Poisson distribution with
intensity $\lambda>0$ is the two-point measure
\[
\Pi_\lambda^\uplus
=
\frac{1}{1+\lambda}\delta_0
+
\frac{\lambda}{1+\lambda}\delta_{\lambda+1}.
\]

\section{Main results}\label{Seq:noncommutativelawofrareevents}

In this section, we present our main results concerning the law of rare events and binomial approximations in the non-commutative setting. Throughout, we denote by \( \bm{\mu} \) a tuple of probability measures in \( \mathcal{P}(\mathbb{R}) \), written as \( \bm{\mu} = (\mu_1, \dots, \mu_n) \), where each \( \mu_i \in \mathcal{P}(\mathbb{R}) \). For notational convenience, we introduce the shorthand
\begin{align}
\bm{\mu}^{\circledast} := \mu_1 \circledast \cdots \circledast \mu_n,
\end{align}
which will prove particularly useful when dealing with binomial approximations.  The remainder of this section is divided into two parts: the first focuses on results related to non-commutative Poisson approximations (namely, the non-commutative law of rare events), and the second is devoted to non-commutative binomial approximation results. 

\subsection{Non-commutative law of rare events}
We begin this section with the quantitative law of rare events in its simplest form, stated formally below

\begin{theorem}\label{eq:lawofrareevents}
Consider Bernoulli probability measures
$\mu_1,\dots,\mu_n$ with success probabilities
$p_1,\dots,p_n$, and define
\[
\lambda:=\sum_{k=1}^n p_k.
\]

For the tensor convolution,
\begin{align}
\mathbbm{d}_{W^1}
\left(
\bm{\mu}^{*},
\Pi_{\lambda}^{*}
\right)
&\leq
\sum_{k=1}^n p_k^2.
\end{align}

For the Boolean convolution,
\begin{align}
d_{W^1}
\left(
\bm{\mu}^{\uplus},
\Pi_{\lambda}^{\uplus}
\right)
&\leq
\frac{2}{1+\lambda}
\sum_{k=1}^n p_k^2
\leq
2\sum_{k=1}^n p_k^2.
\end{align}

For the free convolution,
\begin{align}
d_{W^1}
\left(
\bm{\mu}^{\boxplus},
\Pi_{\lambda}^{\boxplus}
\right)
&\leq
3\sum_{k=1}^n p_k^{3/2}.
\end{align}
\end{theorem}
\begin{remark}
We do not claim that these rates are optimal. In
particular, it remains open whether the exponent $3/2$ in
the free case can be replaced by $2$.
\end{remark}

\noindent In the tensor and free cases, the Bernoulli assumption can be relaxed at the cost of an additional term measuring the first-moment contribution outside the set \(\{0,1\}\).

\begin{corollary}\label{eq:lawofrareeventsgeneral}
Let
\[
\mu_1,\dots,\mu_n\in\mathcal{P}_c(\R),
\]
and define
\[
p_k:=\mu_k[\{1\}],
\qquad
\lambda:=\sum_{k=1}^n p_k.
\]

For the tensor convolution,
\begin{align}
\mathbbm{d}_{W^1}
\left(
\bm{\mu}^{*},
\Pi_{\lambda}^{*}
\right)
&\leq
\sum_{k=1}^n p_k^2
+
\sum_{k=1}^{n}
\int_{\R\setminus\{0,1\}}|x|\,\mu_k(dx).
\end{align}

For the free convolution,
\begin{align}
d_{W^1}
\left(
\bm{\mu}^{\boxplus},
\Pi_{\lambda}^{\boxplus}
\right)
&\leq
3\sum_{k=1}^n p_k^{3/2}
+
\sum_{k=1}^{n}
\int_{\R\setminus\{0,1\}}|x|\,\mu_k(dx).
\end{align}
\end{corollary}

\subsection{Non-commutative binomial approximations and arithmetic averages}
Throughout this subsection, $\circledast\in\{*,\boxplus\}$. We now deal with the approximations discussed in the previous section, but when the target distribution is binomial instead of Poisson. Once more, we assume that the measures \( \mu_i \) correspond to Bernoulli distributions with success probabilities \( p_1, \dots, p_n \). For $q\in[0,1]$ and $n\geq1$, let
\[
\rho_q:=(1-q)\delta_0+q\delta_1.
\]
We define the $\circledast$-binomial distribution with
success probability $q$ and sample size $n$ by
\[
\nu_B^{q,n}:=\rho_q^{\circledast n}.
\]
where $\rho_i(dx) =(1-q)\delta_{0}+q\delta_1$, for $q$ defined as the mean of the $p_i$'s. A robust formulation of approximations to $\nu_B^{q,n}$ in the tensorial case can be consulted in \cite{Ehm1991}, where binomial approximations to Poisson-Binomial distributions are addressed by means of Stein's method.\\

The following proposition is a direct consequence of Lemma~\ref{forthbefbinomprop}.

\begin{Proposition}\label{eq:binomial}
Let
\[
q:=\frac{1}{n}\sum_{j=1}^{n}p_j,
\]
and let $\nu_B^{q,n}$ be the corresponding
$\circledast$-binomial distribution. Then
\begin{align}\label{eq:binomapprox}
d_{W^1}
\left(
\bm{\mu}^{\circledast},
\nu_B^{q,n}
\right)
&\leq
\sum_{k=1}^{n}|p_k-q|.
\end{align}
\end{Proposition}

\begin{remark}
Combining Proposition \ref{eq:binomial} with the
Cauchy-Schwarz inequality, we obtain
\[
d_{W^1}
\left(
\bm{\mu}^{\circledast},
\nu_B^{q,n}
\right)
\leq
n\operatorname{Var}[p_{I_n}]^{1/2},
\]
where $I_n$ is uniformly distributed on
$\{1,\dots,n\}$.
\end{remark}

\noindent The above result is a particular case of the following lemma, which allows comparisons of $\circledast$-convolutions of measures with the corresponding convolution of their arithmetic average. In the sequel, for a family $\mu_1,\dots, \mu_n$, we will denote by $\overline{\bm{\mu}}$ its average 
\begin{align*}
\overline{\bm{\mu}}
  &:=\frac{1}{n}\sum_{k=1}^n\mu_k.	
\end{align*}
Proposition \ref{eq:binomial} is a direct consequence of the following lemma, which follows by applying  Lemma \ref{Lemma:convolutioninequality} with $\rho_k=\overline{\bm\mu}$ for every $k$.

\begin{lemma}\label{forthbefbinomprop}
Let $\mu_1,\dots,\mu_n\in\mathcal{P}_c(\R)$. Then
\begin{align*}
d_{W^1}
\left(
\bm{\mu}^{\circledast},
\overline{\bm{\mu}}^{\circledast n}
\right)
&\leq
\sum_{k=1}^n
d_{W^1}(\mu_k,\overline{\bm{\mu}}).
\end{align*}
\end{lemma}

The proof of Proposition~\ref{eq:binomial} follows from Lemma~\ref{forthbefbinomprop} by observing that in the particular case where the $\mu_k$ are Bernoulli with success probability $p_k$, then $\overline{\bm{\mu}}$ is Bernoulli with parameter $(p_1+\cdots+p_n)/n$, which yields 
 \begin{align*}
 d_{W^{1}}(\mu_k,\overline{\bm{\mu}})
   &\leq |p_k-(p_1+\cdots+p_n)/n|,	
 \end{align*}
thus implying \eqref{eq:binomapprox}.

\section{Proofs}\label{Proofs}
\noindent In this section we present the proofs of Theorem \ref{eq:lawofrareevents} and Corollary \ref{eq:lawofrareeventsgeneral}.

\subsection{Proof of Theorem \ref{eq:lawofrareevents}}
\noindent The tensorial case follows from the stability of the classical Wasserstein distance under convolution. Indeed, using the one-dimensional representation of the
Wasserstein distance in terms of distribution functions, a direct computation gives, for every $k$,
\[
\mathbbm{d}_{W^1}
\left(
\mu_k,
\Pi_{p_k}^{*}
\right)
=
2\left(p_k-1+e^{-p_k}\right)
\leq
p_k^2.
\]
Consequently,
\begin{align*}
\mathbbm{d}_{W^1}
\left(
\bm{\mu}^{*},
\Pi_{\lambda}^{*}
\right)
&\leq
\sum_{k=1}^n
\mathbbm{d}_{W^1}
\left(
\mu_k,
\Pi_{p_k}^{*}
\right) \leq
\sum_{k=1}^n p_k^2.
\end{align*}
We now consider the free and Boolean cases.\\

\noindent\textit{Step I}\\
We first consider the free case, so throughout this step we set
$\circledast=\boxplus$. For each $k=1,\dots,n$, define
\[
\eta_k:=\Pi_{p_k}^{\circledast},
\quad\quad\quad\quad\quad\quad
\lambda:=\sum_{k=1}^n p_k,
\]
as well as $\bm{\eta}:=(\eta_1,\dots,\eta_n)$ and $\nu:=\Pi_{\lambda}^{\circledast}$. 
By additivity of free cumulants,
\[
\eta_1\circledast\cdots\circledast\eta_n
=
\Pi_{\lambda}^{\circledast},
\]
and therefore $\bm{\eta}^{\circledast}=\nu$. Define the Lindeberg-type interpolation
\begin{align}\label{eq:lindeberginterpolation}
\rho_{k,n}
  &:=
  \mu_1\circledast\cdots\circledast\mu_k
  \circledast
  \eta_{k+1}\circledast\cdots\circledast\eta_n,
\end{align}
for $k=1,\dots,n-1$, and set
\[
\rho_{n,n}:=\bm{\mu}^{\circledast},
\qquad
\rho_{0,n}:=\bm{\eta}^{\circledast}.
\]
By the triangle inequality,
\begin{align}\label{eq:boundfWsums}
d_{W^{1}}(\bm{\mu}^{\circledast},\bm{\eta}^{\circledast})
  &\leq
  \sum_{k=0}^{n-1}
  d_{W^1}(\rho_{k,n},\rho_{k+1,n}).
\end{align}
For $k=0,\dots,n-1$, Lemma~\ref{Lemma:convolutioninequality} yields
\begin{align}\label{eq:boundfWregtospec}
d_{W^1}(\rho_{k,n},\rho_{k+1,n})
  &\leq
  d_{W^1}(\mu_{k+1},\eta_{k+1})
  \leq
  \mathbbm{d}_{W^1}(\mu_{k+1},\eta_{k+1}).
\nonumber
\end{align}
It therefore suffices to bound
\[
\mathbbm{d}_{W^1}(\mu_k,\eta_k).
\]
The case $p_k=0$ is immediate, since then
\[
\mu_k=\eta_k=\delta_0.
\]
We may therefore assume that $p_k>0$. Let
$f:\R\rightarrow\R$ be a $1$-Lipschitz function. Since adding a constant to $f$ does not change the difference of the corresponding integrals, we may assume that $f(0)=0$. Since $\mu_k$ is Bernoulli with success probability $p_k$ and
$m_1[\eta_k]=p_k$, we have
\begin{align*}
\int_{\R}f(x)\eta_k(dx)
-
\int_{\R}f(x)\mu_k(dx)
&=
\int_{\R_+}
\bigl(f(x)-xf(1)\bigr)\eta_k(dx).
\end{align*}
For every $x\geq0$,
\[
|f(x)-xf(1)|
\leq
2x|x-1|.
\]
Indeed, if $0\leq x\leq1$, then
\begin{align*}
|f(x)-xf(1)|
&=
|x(f(x)-f(1))+(1-x)f(x)|
\leq
x(1-x)+(1-x)x
=
2x(1-x).
\end{align*}
If $x\geq1$, then
\begin{align*}
|f(x)-xf(1)|
&\leq
|f(x)-f(1)|+(x-1)|f(1)|
\leq
2(x-1)
\leq
2x(x-1).
\end{align*}
It follows from the Kantorovich-Rubinstein duality that
\begin{align}\label{eq:technicalsplitoneprev2}
\mathbbm{d}_{W^1}(\mu_k,\eta_k)
&\leq
2\int_{\R_+}x|x-1|\eta_k(dx).
\nonumber
\end{align}
Applying the Cauchy--Schwarz inequality, we obtain
\begin{align*}
\mathbbm{d}_{W^1}(\mu_k,\eta_k)
&\leq
2\sqrt{m_1[\eta_k]}
\left(
\int_{\R}x|x-1|^2\eta_k(dx)
\right)^{1/2}\\
&=
2\sqrt{m_1[\eta_k]}
\left(
m_3[\eta_k]-2m_2[\eta_k]+m_1[\eta_k]
\right)^{1/2}.
\end{align*}
Since $\eta_k$ is the free Poisson distribution with intensity $p_k$,
\[
\kappa_r^{\boxplus}[\eta_k]=p_k
\]
for every $r\geq1$. The third-order moment-cumulant identity therefore gives
\[
m_3[\eta_k]-2m_2[\eta_k]+m_1[\eta_k]
=
p_k^3+p_k^2.
\]
It follows that
\begin{align*}
\mathbbm{d}_{W^1}(\mu_k,\eta_k)
&\leq
2p_k^{3/2}(1+p_k)^{1/2}
\leq
3p_k^{3/2}.
\end{align*}
Combining this estimate with \eqref{eq:boundfWsums}, we conclude that
\[
d_{W^{1}}(\bm{\mu}^{\boxplus},\Pi_{\lambda}^{\boxplus})
\leq
3\sum_{k=1}^{n}p_k^{3/2}.
\]

\noindent\textit{Step II}\\
We now consider the Boolean case. By
Lemma~\ref{lem:boolean-bernoulli-support},
\[
\operatorname{supp}(\bm{\mu}^{\uplus})
\subseteq
[0,1+\lambda].
\]
Moreover,
\[
\kappa_1^{\uplus}[\mu_k]=p_k
\qquad\text{and}\qquad
\kappa_2^{\uplus}[\mu_k]=p_k-p_k^2.
\]
By additivity of Boolean cumulants,
\[
\kappa_1^{\uplus}[\bm{\mu}^{\uplus}]
=
\lambda
\qquad\text{and}\qquad
\kappa_2^{\uplus}[\bm{\mu}^{\uplus}]
=
\lambda-\sum_{k=1}^n p_k^2.
\]
so that 
\[
m_1[\bm{\mu}^{\uplus}]
=
\lambda
\qquad\text{and}\qquad
m_2[\bm{\mu}^{\uplus}]
=
\lambda^2+\lambda-\sum_{k=1}^n p_k^2.
\]
Let $X$ be a random variable with distribution
$\bm{\mu}^{\uplus}$, and let $U$ be a uniform random
variable on $[0,1]$, independent of $X$. Define
\[
Y
:=
(1+\lambda)
\mathbbm{1}_{\left\{
U\leq X/(1+\lambda)
\right\}}.
\]
By Lemma~\ref{lem:boolean-bernoulli-support},
\[
0\leq X\leq1+\lambda
\]
almost surely, and therefore  \(X/(1+\lambda)\) defines a conditional probability given $X$. By a direct computation, we have that 
\begin{align*}
\Pb[Y=1+\lambda]
&=
\E\big[
\Pb\big[
U\leq X/(1+\lambda)
\, |\,X
\big]
\big]
=
\frac{\E[X]}{1+\lambda}
=
\frac{\lambda}{1+\lambda}.
\end{align*}
Consequently, $\Pb[Y=0] = 1/(1+\lambda),
$
and hence
\[
\text{Law}(Y)
=
\frac{1}{1+\lambda}\delta_0
+
\frac{\lambda}{1+\lambda}\delta_{1+\lambda}
=
\Pi_{\lambda}^{\uplus}.
\]
Let $\pi$ denote the joint distribution of $(X,Y)$. Its first marginal is $\bm{\mu}^{\uplus}$ and its second marginal is $\Pi_{\lambda}^{\uplus}$. Therefore, $\pi$ is a transport plan for $(\bm{\mu}^{\uplus},
\Pi_{\lambda}^{\uplus})$. Observe that 
\begin{align*}
\E[|X-Y|\mid X=x]
&=
\big(1-x/(1+\lambda)\big)x
+
\frac{x}{1+\lambda}(1+\lambda-x)
=
\frac{2x(1+\lambda-x)}{1+\lambda}.
\end{align*}
It follows that
\begin{align*}
\mathbbm{d}_{W^1}
\left(
\bm{\mu}^{\uplus},
\Pi_{\lambda}^{\uplus}
\right)
&\leq
\E[|X-Y|]
=
\frac{2}{1+\lambda}
\E\left[
X(1+\lambda-X)
\right],
\end{align*}
so that 
\begin{align*}
\mathbbm{d}_{W^1}
\left(
\bm{\mu}^{\uplus},
\Pi_{\lambda}^{\uplus}
\right)
&\leq
\frac{2}{1+\lambda}
\left(
(1+\lambda)m_1[\bm{\mu}^{\uplus}]
-
m_2[\bm{\mu}^{\uplus}]
\right)
=
\frac{2}{1+\lambda}
\sum_{k=1}^n p_k^2
\leq
2\sum_{k=1}^n p_k^2.
\end{align*}
The result now follows from
\[
d_{W^1}
\left(
\bm{\mu}^{\uplus},
\Pi_{\lambda}^{\uplus}
\right)
\leq
\mathbbm{d}_{W^1}
\left(
\bm{\mu}^{\uplus},
\Pi_{\lambda}^{\uplus}
\right).
\]

\subsection{Proof of Corollary
\ref{eq:lawofrareeventsgeneral}}

For each $k=1,\dots,n$, define $p_k:=\mu_k[\{1\}]$ and let
\[
\beta_k
:=
(1-p_k)\delta_0+p_k\delta_1.
\]
Set
\[
\bm{\beta}:=(\beta_1,\dots,\beta_n),
\qquad
\lambda:=\sum_{k=1}^n p_k.
\]
Define $T:\R\to\{0,1\}$ by
\[
T(x)
:=
\begin{cases}
1, & x=1,\\
0, & x\neq1.
\end{cases}
\]
The pushforward of $\mu_k$ under $T$ is $\beta_k$. Hence
the joint distribution of $(X,T(X))$, where
$X$ has distribution $\mu_k$, is a tensor transport plan
between $\mu_k$ and $\beta_k$. Consequently,
\begin{align*}
\mathbbm{d}_{W^1}(\mu_k,\beta_k)
&\leq
\int_{\R}|x-T(x)|\,\mu_k(dx)
=
\int_{\R\setminus\{0,1\}}|x|\,\mu_k(dx).
\end{align*}

For the tensor convolution, the stability of the classical
Wasserstein distance under convolution gives
\begin{align*}
\mathbbm{d}_{W^1}
\left(
\bm{\mu}^{*},
\bm{\beta}^{*}
\right)
&\leq
\sum_{k=1}^n
\mathbbm{d}_{W^1}(\mu_k,\beta_k)
\leq
\sum_{k=1}^n
\int_{\R\setminus\{0,1\}}|x|\,\mu_k(dx).
\end{align*}
By Theorem~\ref{eq:lawofrareevents},
\[
\mathbbm{d}_{W^1}
\left(
\bm{\beta}^{*},
\Pi_{\lambda}^{*}
\right)
\leq
\sum_{k=1}^n p_k^2.
\]
The tensor estimate follows from the triangle inequality. For the free convolution, Lemma~\ref{Lemma:convolutioninequality}
and the inequality
\[
d_{W^1}(\mu_k,\beta_k)
\leq
\mathbbm{d}_{W^1}(\mu_k,\beta_k)
\]
give
\begin{align*}
d_{W^1}
\left(
\bm{\mu}^{\boxplus},
\bm{\beta}^{\boxplus}
\right)
&\leq
\sum_{k=1}^n
d_{W^1}(\mu_k,\beta_k)
\leq
\sum_{k=1}^n
\int_{\R\setminus\{0,1\}}|x|\,\mu_k(dx).
\end{align*}
By Theorem~\ref{eq:lawofrareevents},
\[
d_{W^1}
\left(
\bm{\beta}^{\boxplus},
\Pi_{\lambda}^{\boxplus}
\right)
\leq
3\sum_{k=1}^n p_k^{3/2}.
\]
The free estimate follows from the triangle inequality.

\appendix 

\section{Technical lemmas}\label{secLlemapp}
\noindent In this section we present a technical lemma that was utilized in the proof of Theorem \ref{eq:lawofrareevents}
\begin{lemma}\label{lem:boolean-bernoulli-support}
For $1\leq k\leq n$, define 
\[
\beta_k=(1-p_k)\delta_0+p_k\delta_1,
\]
where $0\leq p_k\leq 1$, and set $\lambda:=p_1+\cdots+p_n$, as well as 
\[
\mu:=\beta_1\uplus\cdots\uplus\beta_n.
\]
Then
\[
\operatorname{supp}(\mu)\subseteq[0,1+\lambda].
\]
\end{lemma}

\begin{proof}
If $\lambda=0$, then $p_k=0$ for every $k$, and hence
$\mu=\delta_0$. The result is therefore immediate. We may
thus assume that $\lambda>0$. For a compactly supported probability measure $\gamma$, define
\[
G_\gamma(z)
:=
\int_{\R}\frac{1}{z-x}\,\gamma(dx),
\qquad
F_\gamma(z)
:=
\frac{1}{G_\gamma(z)},
\]
as well as 
\[
K_\gamma(z)
:=
z-F_\gamma(z).
\]
Recall that Boolean convolution satisfies
\[
K_{\gamma\uplus\rho}(z)
=
K_\gamma(z)+K_\rho(z).
\]
We can easily check that for every $1\leq k\leq n$,
\[
F_{\beta_k}(z)
=
\frac{z(z-1)}{z-1+p_k},
\]
and therefore
\[
K_{\beta_k}(z)
=
\frac{p_kz}{z-1+p_k}.
\]
From here it follows that 
\[
F_\mu(z)
=
z\big(
1-\sum_{k=1}^n
 p_k/(z-1+p_k)
\big).
\]
For $x>1$, define
\[
H(x)
:=
1-\sum_{k=1}^n
\frac{p_k}{x-1+p_k}.
\]
The function $H$ is strictly increasing on $(1,\infty)$,
since
\[
H'(x)
=
\sum_{k=1}^n
\frac{p_k}{(x-1+p_k)^2}
>
0.
\]
Moreover,
\[
H(1+\lambda)
=
1-\sum_{k=1}^n
\frac{p_k}{\lambda+p_k}
\geq
1-\sum_{k=1}^n\frac{p_k}{\lambda}
=
0.
\]
Hence $H(x)>0$ for every $x>1+\lambda$, and consequently
\[
F_\mu(x)=xH(x)>0,
\]
for $x>1+\lambda$. Since $F_\mu$ is a rational function, so is $G_\mu(z)=1/F_\mu(z)$, and consequently, $\mu$ is a finite atomic measure with each support point being a pole of $G_\mu$. It follows that $\mu$ is finitely supported and  the support points of $\mu$ are precisely the poles of
$G_\mu$, which are precisely  the zeros of $F_\mu=1/G_\mu$. Since $F_\mu(x)>0$ for every $x>1+\lambda$, then  $G_\mu$ has no poles in the region $x>1+\lambda$, and hence no support point of $\mu$ can lie above $1+\lambda$.\\

\noindent To rule out negative elements in the support of $\mu$, we take $x<0$, so that $x-1+p_k<0$. We then observe that 
\[
H(x)
=
1-\sum_{k=1}^n
\frac{p_k}{x-1+p_k}
\geq 1.
\]
Consequently, $F_\mu(x)=xH(x)<0$, thus implying that  $F_\mu$ has no zeros on $(-\infty,0)$. By an argument analogous as before, we deduce that  $\mu$ has no support points outside $[0,1+\lambda]$. The result follows from here
\end{proof}

\noindent \textbf{Acknowledgements}\\
Arturo Jaramillo Gil was supported by
the grant CBF2023-2024-2088.

\bibliographystyle{plain}
\bibliography{bibibi}

\end{document}